\documentclass{amsart}
\usepackage[margin=1.2in]{geometry}
\usepackage{amsmath,amssymb,amsthm,graphicx,amscd,thmtools}
\usepackage[all,cmtip]{xy}
\usepackage{tikz}

\theoremstyle{plain}
\newtheorem*{main1}{Main Theorem I}
\newtheorem*{main2}{Main Theorem II}
\newtheorem*{thm*}{Theorem}
\newtheorem*{lm*}{Lemma}
\newtheorem{thm}{Theorem}[subsection]
\newtheorem{lm}[thm]{Lemma}
\newtheorem{cor}[thm]{Corollary}
\newtheorem*{cor*}{Corollary}
\newtheorem{prop}[thm]{Proposition}

\newtheorem{que}[thm]{Question}

\theoremstyle{definition}
\newtheorem*{de*}{Definition}
\declaretheorem[sibling=thm,name=Definition,qed={$\diamondsuit$}]{de}
\declaretheorem[sibling=thm,name=Example,qed={$\diamondsuit$}]{ex}
\declaretheorem[sibling=thm,name=Remark,qed={$\diamondsuit$}]{re}
\newtheorem*{ex*}{Example}

\newcommand{\NN}{\mathbb{N}}

\newcommand{\RR}{\mathbb{R}}

\newcommand{\PP}{\mathbb{P}}

\DeclareMathOperator{\im}{im}
\DeclareMathOperator{\cha}{char}
\DeclareMathOperator{\id}{id}

\renewcommand{\phi}{\varphi}
\DeclareMathOperator{\str}{str}
\DeclareMathOperator{\rk}{rk}

\DeclareMathOperator{\End}{End}
\DeclareMathOperator{\GL}{GL}

\DeclareMathOperator{\Hom}{Hom}
\DeclareMathOperator{\Sh}{Sh}

\DeclareMathOperator{\Diag}{diag}
\newcommand{\Wedge}{\bigwedge\nolimits}

\renewcommand{\Vec}{\mathbf{Vec}}

\begin{document}

\title{Universality of high-strength tensors}

\author{Arthur Bik}
\address{MPI for Mathematics in the Sciences, Leipzig, Germany}
\email{arthur.bik@mis.mpg.de}

\author{Alessandro Danelon}
\address{Eindhoven University of Technology, The Netherlands}
\email{a.danelon@tue.nl}

\author{Jan Draisma}
\address{University of Bern, Switzerland, and Eindhoven University of Technology, The Netherlands}
\email{jan.draisma@math.unibe.ch}

\author{Rob H. Eggermont}
\address{Eindhoven University of Technology, The Netherlands}
\email{r.h.eggermont@tue.nl}

\let\thefootnote\relax
\footnotetext{\hspace*{-14pt}
\begin{minipage}{.05\textwidth}
\includegraphics[width=\textwidth]{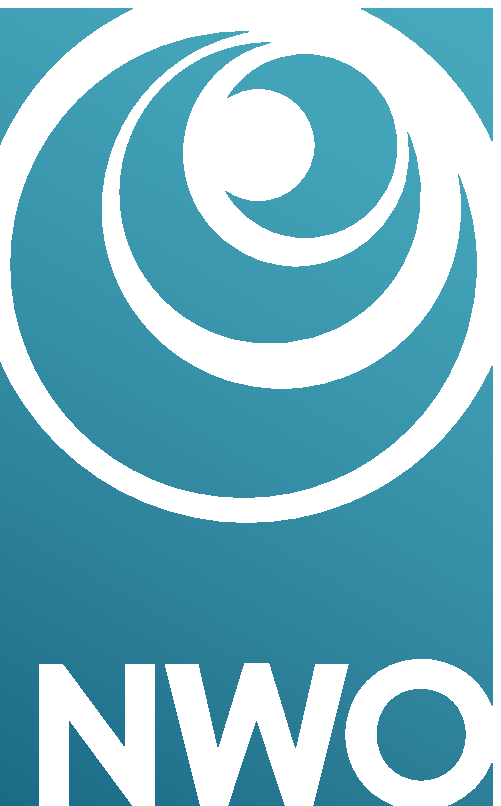}
\end{minipage}
\begin{minipage}{416pt}
AD is supported by JD's Vici grant 639.033.514 from
the Netherlands Organisation for scientific research (NWO).
JD is partially supported by Vici grant 639.033.514
from the NWO and by project grant 200021\textunderscore 191981 from the
Swiss National Science Foundation (SNSF). RE was supported by Veni grant 016.Veni.192.113 from the NWO.
\end{minipage}}

\maketitle

%\begin{center}
%{\em To Bernd Sturmfels, on the occasion of his 60th birthday.}
%\end{center}

\begin{abstract}
A theorem due to Kazhdan and Ziegler implies that, by substituting linear
forms for its variables, a homogeneous polynomial of sufficiently high
strength specialises to any given polynomial of the same degree in a
bounded number of variables. Using entirely different techniques,
we extend this theorem to arbitrary polynomial functors.  As a
corollary of our work, we show that specialisation induces a quasi-order
on elements in polynomial functors, and that among the 
elements with a dense orbit there are unique smallest and
largest equivalence classes in this quasi-order.
\end{abstract}

\section{Introduction}

Let $K$ be an algebraically closed field of characteristic
$0$. For partitions $\lambda$ of integers $d\geq 1$, denoted
as $\lambda\vdash d$, we consider the corresponding Schur 
functors $S_{\lambda}$. We refer the reader to
\cite{G:polyreps} or \cite[Lecture~6]{FH:first_course} for
an introduction to these objects. For a tuple
$\underline{\lambda}=[\lambda_1,\ldots,\lambda_k]$ of
partitions $\lambda_i\vdash d_i\geq 1$, we denote
$S_{\lambda_1}\oplus\cdots\oplus S_{\lambda_k}$ by
$S_{\underline{\lambda}}$. For finite-dimensional vector
spaces $V,W$ and a linear map $\phi:V \to W$, we get 
a linear map
\[
S_{\underline{\lambda}}(\phi)\colon S_{\underline{\lambda}}(V)\to S_{\underline{\lambda}}(W)
\]
that depends polynomially on $\phi$ and satisfies 
$S_{\underline{\lambda}}(\id_V)=\id_{S_{\underline{\lambda}}(V)}$ and $S_{\underline{\lambda}}(\phi\circ\psi)=S_{\underline{\lambda}}(\phi)\circ S_{\underline{\lambda}}(\psi)$ whenever the former makes sense. In particular, taking
$V=W$ and restricting our attention to invertible $\phi$, we
find that $S_{\underline{\lambda}}(V)$ is
a polynomial representation of the group $\GL(V)$. 

\begin{ex*}
For $\lambda=(d)$, $S_\lambda(V)=S^d V$, the
$d$-th symmetric power of $V$. If $x_1,\ldots,x_n$ is a
basis of $V$, $S_{(d)}(V)$ is the space of homogeneous
polynomials of degree $d$ in $x_1,\ldots,x_n$.
\end{ex*}

For two tuples $\underline{\lambda},\underline{\nu}$ of partitions, we
write $\underline{\nu}\lessdot\underline{\lambda}$ when the number of
occurrences of every partition $\mu\vdash d$ in $\underline{\nu}$ is at
most the number of occurrences of $\mu$ in $\underline{\lambda}$, where
$d$ is the maximal integer for which these numbers differ for some $\mu$.

\begin{ex*}
We have $[(1),(1),(1,1),(3)] \lessdot [(2),(3),(2,1)]$. 
\hfill$\diamondsuit$
\end{ex*}

Let $\underline{\lambda}$ be a tuple of partitions of
positive integers. The following dichotomy is our first main result.

\begin{main1}
Let $\mathcal{P}$ be a property that, for each
finite-dimensional vector space $V$, can be satisfied by some elements of $S_{\underline{\lambda}}(V)$. Assume that $S_{\underline{\lambda}}(\phi)(f)\in S_{\underline{\lambda}}(W)$ satisfies $\mathcal{P}$ for every element $f\in  S_{\underline{\lambda}}(V)$ satisfying $\mathcal{P}$ and every linear map $\phi\colon V\to W$. Then either $\mathcal{P}$ is satisfied by all elements of $S_{\underline{\lambda}}(V)$ for all $V$ or else all elements satisfying $\mathcal{P}$ come from simpler spaces $S_{\underline{\mu}}(V)$ for finitely many tuples $\underline{\mu}\lessdot\underline{\lambda}$.
\end{main1}

We define later what it means to ``come from
$S_{\underline{\mu}}(V)$''; for a more precise formulation of the theorem,
see Theorem~\ref{thm:MainI}. When $\underline{\lambda}$ consists of one
partition, the second case in the theorem says that elements
satisfying $\mathcal{P}$ have bounded strength in the
following sense. 

\begin{de*}
The strength of an element $f\in S_{\lambda}(V)$ with $\lambda \vdash d$ is the minimal integer $k\geq0$ such that there exists an expression
\[
f= \alpha_1(g_1, h_1)+\ldots+ \alpha_k(g_k, h_k)
\]
where $\mu_i \vdash d_i$, $\nu_i \vdash e_i$ with $d_i, e_i
< d$, the $\alpha_i \colon S_{\mu_i}(V)\oplus S_{\nu_i}(V)\to S_{\lambda}(V)$ are $\GL(V)$-equivariant bilinear maps and the $g_i\in S_{\mu_i}(V), h_i\in S_{\nu_i}(V)$ are elements.
\hfill$\diamondsuit$
\end{de*}

In Definition~\ref{de:Strength2} we will give a broader definition
that is equivalent to the one above for tuples consisting of
a single partition. The
definition above and Definition~\ref{de:Strength2} extend the strength
of polynomials and of tuples of polynomials, respectively. Strength of
polynomials plays a key role in the resolution of Stillman's conjecture
by Ananyan-Hochster \cite{AH:stillmanconj} and in recent work by
Kazhdan-Ziegler \cite{KZ:strength1,KZ:strength2}.  Main
Theorem I is an extension (in characteristic zero) of \cite[Theorem
1.9]{KZ:strength2} for homogeneous polynomials, which is the case
where $\underline{\lambda}$ is a single partition with a single row.

\medskip

Next, denote the inverse limit of the spaces $S_{\underline{\lambda}}(K^n)$ mapping to each other via $S_{\underline{\lambda}}$ applied to the projection maps $K^{n+1}\to K^n$ by $S_{\underline{\lambda},\infty}$. This space comes with the action of the direct limit $\GL_{\infty}$ of the groups $\GL_n$ mapping into each other via the maps $g\mapsto\Diag(g,1)$. It also comes with a topology induced by the Zariski topologies on $S_{\underline{\lambda}}(K^n)$, which we again call the Zariski topology.

\begin{cor*}[Corollary~\ref{cor:surjective_infty}]
Suppose that the orbit $\GL_\infty\!\cdot  p$ is Zariski dense in $S_{\underline{\lambda},\infty}$. Then for each integer $n\geq1$, the image of $\GL_\infty\!\cdot p$ in $S_{\underline{\lambda}}(K^n)$ is all of $S_{\underline{\lambda}}(K^n)$.
\end{cor*}

The second goal of this paper is to bring some order in the
(typically uncountable) set of elements with dense
$\GL_{\infty}$-orbits. For elements $p,q\in S_{\underline{\lambda},\infty}$, we write $p\preceq q$ when $q$ specialises to $p$; see \S\ref{ss:linendo}-8 for details.

\begin{ex*}
When $\lambda=(d)\vdash d$, the space
$S_{\lambda,\infty}$ consists of infinite degree-$d$ forms
in variables $x_1,x_2,\ldots$. We have $p\preceq q$ if and
only if $p=q(\ell_1,\ell_2,\ldots)$ where
$\ell_1,\ell_2,\ldots$ are infinite linear forms such that
for all $i \geq 1$, the variable $x_i$ occurs in only finitely many forms 
$\ell_j$; this ensures that $q(\ell_1,\ell_2,\ldots)$ is
a well-defined infinite form of degree $d$.
\hfill$\diamondsuit$
\end{ex*}

Our second main result is the following theorem.

\begin{main2}[Theorem~\ref{thm:MainII}]
Let $\underline{\lambda}$ be a tuple of partitions, all  of
the same integer $d\geq1$. 
There exist elements $p,r\in
S_{\underline{\lambda},\infty}$, each with a dense
$\GL_{\infty}$-orbit, such that $p\preceq q\preceq r$ for all other $q\in S_{\underline{\lambda},\infty}$ with a dense $\GL_{\infty}$-orbit.
\end{main2}

\subsection*{Structure of the paper}

In \S\ref{sec:def&res}, we introduce all relevant
definitions and restate our main results in more precise
terms. Also, while our main results require characteristic
zero, some of our theory is developed in arbitrary
characteristic. In \S\ref{sec:ProofI}, we prove Main Theorem I.  In \S\ref{sec:ProofII}, we prove Main Theorem II by constructing minimal $p$ and maximal $r$. Finally, we end with some examples in \S\ref{sec:examples}.

\subsection*{Acknowledgments}

We thank Andrew Snowden, who first pointed out to us the
action of the monoid~$E$ on $P_\infty$ and asked about
its orbit structure there.

\section{Definitions and main results} \label{sec:def&res}

Fix a field $K$. In our main results we will assume that $K$ is
algebraically closed and of characteristic zero, but for now we make
no such assumption.

\subsection{Strength}

\begin{de} \label{de:Strength}
Let $n\geq 1$ be an integer and let $f \in K[x_1,\ldots,x_n]_d$ be a
homogeneous polynomial of degree $d\geq 2$. Then the {\em strength} of $f$,
denoted $\str(f)$, is the minimal integer $k\geq0$ such that there exists an
expression
\[
f= g_1\cdot h_1+\ldots+ g_k\cdot h_k
\]
where $g_i \in K[x_1,\ldots,x_n]_{d_i}$ and $h_i \in K[x_1,\ldots,x_n]_{d-d_i}$ for some integer $0<d_i<d$ for each $i\in[k]$.
\end{de}

The strength of polynomials plays a key role in the resolution of Stillman's
conjecture by Ananyan-Hochster \cite{AH:stillmanconj,AH:stillmanbounds},
the subsequent work by Erman-Sam-Snowden \cite{ESS:bigpoly,
ESS:strengthhartshorneconj,ESS:bigpoly2} and in Kazhdan-Ziegler's work \cite{KZ:strength1,KZ:strength2}. Also see \cite{BBOV:strengthnotclosed,
BBOV:strengthgen2,BV:strengthline,BDE:boundedstrength,BO:strengthgen1,
DES:noethcubic} for other recent papers studying strength.

\subsection{Polynomial functors and their maps}

Assume that $K$ is infinite.
Let $\Vec$ be the category of finite-dimensional vector spaces over $K$
with $K$-linear maps.

\begin{de}
A {\em polynomial functor of degree $\leq d$} over $K$ is a functor
$P\colon\Vec \to \Vec$ with the property that for all $U,V \in \Vec$ the map
$P\colon\Hom(U,V) \to \Hom(P(U),P(V))$ is a polynomial map of degree
$\leq d$. A {\em polynomial functor} is a polynomial functor of degree
$\leq d$ for some integer $d<\infty$.
\end{de}

\begin{re}
For finite fields $K$, the correct analogue is that of a
{\em strict} polynomial functor \cite{FS:cohomfingroup}.
\end{re}

Any polynomial functor $P$ is a finite direct sum of its {\em homogeneous}
parts $P_d$, which are the polynomial subfunctors defined by
$P_d(V):=\{p \in P(V) \mid \forall t \in K: P(t \id_V)p=t^d p\}$ for each integer
$d\geq0$. A polynomial functor is called homogeneous of degree $d$ when
it equals its degree-$d$ part.

\begin{ex}
The functor $U \mapsto S^d(U)$ is a homogeneous polynomial functor of degree $d$. If $U$ has basis $x_1,\ldots,x_n$, then
$S^d(U)$ is canonically isomorphic to $K[x_1,\ldots,x_n]_d$. In this
incarnation, linear maps $S^d(\phi)$ for $\phi\colon U \to V$ correspond to
substitutions of the variables $x_1,\ldots,x_n$ by linear forms in
variables $y_1,\ldots,y_m$ representing a basis of $V$.
\end{ex}

Polynomial functors are the ambient spaces in current research on
infinite-di\-men\-sional algebraic geometry \cite{B:thesis,BDE:boundedstrength,
BDES:geometrypolyrep,D:topnoeth}. Polynomial functors
form an Abelian category in which a morphism $\alpha\colon P \to Q$ consists
of a linear map $\alpha_U\colon P(U) \to Q(U)$ for each $U \in \Vec$ such that
for all $U,V\in\Vec$ and all $\phi \in \Hom(U,V)$ the following diagram commutes:
\[ \xymatrix{
P(U) \ar[r]^{\alpha_U} \ar[d]_{P(\phi)} & Q(U) \ar[d]^{Q(\phi)} \\
P(V) \ar[r]_{\alpha_V} & Q(V).
} \]
In characteristic zero, each polynomial functor $P$ is isomorphic, in this
Abelian category, to a direct sum of Schur functors, which can be thought
of as subobjects (or quotients) of the polynomial functors $V \mapsto
V^{\otimes d}$. For that reason, we will informally refer to elements
of $P(V)$ as {\em tensors}.

In addition to the linear morphisms between polynomial functors above,
we may also allow each~$\alpha_U$ to be a {\em polynomial} map $P(U)\to Q(U)$ such that
the diagram commutes. Such an $\alpha$ will be called a {\em polynomial
transformation} from $P$ to $Q$. If $U$ is irrelevant or clear from the
context, we write $\alpha$ instead of $\alpha_U$.

\begin{ex} \label{ex:Strength}
In the context of Definition~\ref{de:Strength}, we set $P:=\bigoplus_{i=1}^k
(S^{d_i} \oplus S^{d-d_i})$ and $Q:=S^d$ and~define~$\alpha$ by
\[
\alpha(g_1,h_1,\ldots,g_k,h_k):=g_1\cdot h_1 + \ldots + g_k\cdot h_k.
\]
This is a polynomial transformation $P \to Q$.
\end{ex}

\begin{ex}
Let $Q,R$ be polynomial functors and $\alpha\colon Q\otimes R\to P$ a linear morphism. Then $(q,r)\mapsto \alpha(q\otimes r)$ defines a {\em bilinear} polynomial transformation $Q\oplus R\to P$.
\end{ex}

Inspired by these examples, we propose the following definition of strength
for elements of homogeneous polynomial functors. We are not sure that
this is the best definition in arbitrary characteristic, so we restrict
ourselves to characteristic zero.

\begin{de}\label{de:Strength2}
Assume that $\cha K=0$. Let $P$ be a homogeneous polynomial functor
of degree $d\geq 2$ and let $V \in \Vec$. The {\em strength} of $p \in P(V)$
is the minimal integer $k\geq0$ such that
\[
p=\alpha_1(q_1,r_1)+\ldots+\alpha_k(q_k,r_k)
\]
where, for each $i\in[k]$, $Q_i,R_i$ are irreducible polynomial functors with positive degrees adding up to $d$, $\alpha_i\colon Q_i \oplus R_i \to P$ is a bilinear polynomial transformation and $q_i \in Q_i(V)$ and $r_i \in R_i(V)$ are tensors.
\end{de}

\begin{re}
Positive degrees of two polynomial functors cannot add up to $1$. So nonzero tensors $p\in P(V)$ of homogeneous polynomial functors $P$ of degree $1$ cannot have finite strength. We say that such tensors $p$ have infinite strength. Note that the strength of $0\in P(V)$ always equals $0$.
\end{re}

\begin{prop}
Assume that $\cha K=0$. For each integer $d\geq 2$, the strength of a polynomial
$f \in S^d(V)$ according to Definition~\ref{de:Strength} equals that
according to Definition~\ref{de:Strength2}.
\end{prop}
\begin{proof}
The inequality $\geq$ follows from the fact that $\alpha_i\colon S^{d_i}
\oplus S^{d-d_i} \to S^d,(g,h)\mapsto g\cdot h$ is a bilinear polynomial transformation. For the inequality $\leq$, suppose
that $\alpha\colon Q \oplus R \to S^d$ is a nonzero bilinear polynomial
transformation, where $Q$ and $R$ are irreducible of degrees $e<d$
and $d-e<d$. So $Q$ and $R$ are Schur functors corresponding to
Young diagrams with $e$ and $d-e$ boxes, respectively, and $Q \otimes
R$ admits a nonzero linear morphism to $S^d$, whose Young diagram is a row of $d$ boxes. The
Littlewood-Richardson rule then implies that the Young diagrams of $Q$
and $R$ must be a single row as well, so that $Q=S^e$ and $R=S^{d-e}$,
and also that there is (up to scaling) a unique morphism $Q \otimes R
= S^e \otimes S^{d-e} \to S^d$, namely, the one corresponding to the
polynomial transformation $(g,h) \mapsto g\cdot h$.
\end{proof}

The strength of a tensor in $P$ quickly becomes very difficult when $P$ is not irreducible.

\begin{ex}
Take $P=(S^d)^{\oplus e}$ for some integer $e\geq 1$. Then the strength of a tuple $(f_1,\ldots,f_e)\in P(V)$ is the minimum number $k\geq0$ such that
\[
f_1,\ldots,f_e\in\mathrm{span}\{g_1,\ldots,g_k\}
\]
where $g_1,\ldots,g_k\in S^d(V)$ are reducible polynomials.
\end{ex}

\begin{ex}
Consider $P=S^2\oplus\Wedge^2$, so that $P(V)=V \otimes V$,
and assume that $K$ is algebraically closed. The only possibilities for $Q$ and $R$ are
$Q(V)=R(V)=V$. The bilinear polynomial transformations $\alpha:Q \oplus R \to P$ are of the form
\[
\alpha(u,v)=a u \otimes v + b v \otimes u=c(u\otimes v+v\otimes u)+d(u\otimes v-v\otimes u)
\]
for certain $a,b,c,d \in K$. We note that $\str(A)=\lceil\rk(A)/2\rceil$ when $A\in S^2(V)$ and $\str(A)=\rk(A)/2$ when $A\in \Wedge^2(V)$. In general, we have
\[
\rk(A)/2,\rk(A+A^\top)/2,\rk(A-A^\top)/2\leq \str(A)\leq \rk(A),\rk(A+A^\top)/2+\rk(A-A^\top)/2
\]
for all $A\in V\otimes V$, where each bound can hold with equality. For example, for the matrix
\[
A=\begin{pmatrix}
0&1\\0&0\\&&\ddots\\&&&0&1\\&&&0&0
\end{pmatrix}
\]
we have $\rk(A+A^\top)/2=\rk(A-A^\top)/2=\str(A)=\rk(A)$.
\end{ex}

\begin{ex}
Again take $P=S^2\oplus\Wedge^2$ and consider $P(K^2)=K^{2\times 2}$. Assume $K$ is algebraically closed. The matrix
\[
A =\begin{pmatrix}1&x\\0&1\end{pmatrix}
\]
clearly has strength $\leq 2$. We will show that $A$ has strength $2$ whenever $x = \pm 2$ and strength $1$ otherwise. In particular, this shows that the subset of $P(K^2)$ of matrices of strength $\leq1$ is not closed.

Suppose $A$ has strength $1$. Then we can write $A$ as $a u \otimes v + b v \otimes u$ with $a, b \in K$ and $v, u \in K^2$. Let $e_1, e_2$ be the standard basis of $K^2$. Without loss of generality, we may assume that $u = e_1 + \lambda e_2$ and $v = e_1 + \mu e_2$ for some $\lambda,\mu\in K$. We get
\begin{align*}
a+b&=1,&a\mu+b\lambda&=x,\\
a\lambda+b\mu&=0,&\lambda\mu&=1.
\end{align*}
Using $\lambda=\mu^{-1}$ and $b = 1-a$, we are left with $a\mu^2 + (1-a) = x\mu$ and $a + (1-a)\mu^2 = 0$. The latter gives us $\mu\neq\pm1$ and $a = \mu^2/(\mu^2-1)$. We get $\mu^2 + 1=x\mu$. Now, if $x\neq\pm2$, then such a $\mu\neq\pm1$ exists. So in this case $A$ indeed has strength $1$. If $x = \pm 2$, the only solution is $\mu = \pm 1$. Hence $A$ has strength $2$ in this case.
\end{ex}

\subsection{Subsets of polynomial functors}

\begin{de}
Let $P$ be a polynomial functor. A {\em subset} of $P$ consists of a
subset $X(U) \subseteq P(U)$ for each $U \in \Vec$ such that for all
$\phi \in \Hom(U,V)$ we have $P(\phi)(X(U)) \subseteq X(V)$. It is {\em
closed} if each $X(U)$ is Zariski-closed in $P(U)$.
\end{de}

\begin{ex}
Fix integers $d\geq 2$ and $k\geq0$.
The elements in $S^d(V)$ of strength $\leq k$
form a subset of $S^d$. This set is closed for $d=2,3$ but
not for $d=4$; see \cite{BBOV:strengthnotclosed}.
\end{ex}

\begin{ex} \label{ex:PosDef}
Take $K=\RR$ and let $X(V)$ be the set of positive semidefinite elements
in $S^2(V)$, i.e., those that are sums of squares of elements
of $V$. Then $X$ is a subset of $S^2$.
\end{ex}

\subsection{Kazhdan-Ziegler's theorem: universality of strength}

\begin{thm}[Kazhdan-Ziegler {\cite[Theorem 1.9]{KZ:strength2}}] \label{thm:KZ}
Let $d \geq2$ be an integer. Assume that $K$ is algebraically closed and
of characteristic $0$ or $>d$.  Let $X$ be a subset of
$S^d$. Then either $X=S^d$ or else there exists an integer $k\geq 0$
such that each polynomial in each $X(U)$ has strength $\leq k$.
\end{thm}

This theorem is a strengthening of \cite[Theorem 4]{BDE:boundedstrength}, where the
additional assumption is that $X$ is closed. The condition that $K$ be
algebraically closed cannot be dropped, e.g.~by Example~\ref{ex:PosDef}:
there is no uniform upper bound on the strength of positive definite
quadratic forms. The condition on the characteristic can also not be
dropped, but see Remark~\ref{re:Char}.

\begin{cor}[Kazhdan-Ziegler, universality of strength]
With the same assumptions on $K$, for every fixed number of variables $m\geq1$ and degree $d\geq2$ there exists an $r\geq0$ such that for any number of variables $n\geq 1$, any polynomial $f \in K[x_1,\ldots,x_n]_d$ of strength $\geq r$ and any polynomial $g \in K[y_1,\ldots,y_m]_d$ there exists a linear variable substitution $x_j
\mapsto \sum_{i} c_{ij} y_i$ under which $f$ specialises to $g$.
\end{cor}
\begin{proof}
For each $U \in \Vec$, define $X(U) \subseteq S^d(U)$ as the set of all $f$
such that the map
\begin{align*}
\Hom(U,K^m) &\to S^d(K^m)\\
\phi &\mapsto S^d(\phi)f
\end{align*}
is {\em not} surjective. A straightforward computation shows that this
is a subset of~$S^d$. It is not all of~$S^d$, because if we take $U$ to
be of dimension $d \cdot \dim S^d(K^m)$, then in $S^d(U)$ we can construct
a sum $f$ of $\dim S^d(K^m)$ squarefree monomials in distinct variables
and specialise each of these monomials to a prescribed multiple of a basis
monomial in $S^d(K^m)$. Hence $f \not \in X(U)$.  By Theorem~\ref{thm:KZ},
it follows that the strength of elements of $X(U)$ is uniformly bounded.
\end{proof}

\subsection{Our generalisation: universality for polynomial functors}

Let $P,Q$ be polynomial functors. We say that $Q$ is smaller than $P$, denoted $Q\lessdot P$, when $P$ and $Q$ are not (linearly) isomorphic and $Q_d$ is a quotient of $P_d$ for the highest degree $d$ where $P_d$ and $Q_d$ are not isomorphic.
We say that a polynomial functor $P$ is pure when $P(\{0\})=\{0\}$.

\begin{re}\label{re:degen+finstrength}
Let $Q\lessdot P$ be polynomial functors and suppose that $P$ is homogeneous of~degree $d>0$. Then $Q_d$ must be a quotient of $P_d$. So we see that $Q\oplus R\lessdot P$ for any polynomial functor $R$ of degree $<d$.
\end{re}

The following is our first main result.

\begin{thm}[Main Theorem I] \label{thm:MainI}
Assume that $K$ is algebraically closed of characteristic zero. Let~$X$ be a subset of a pure polynomial functor $P$ over $K$. Then either $X(U)=P(U)$ for all $U \in \Vec$ or
else there exist finitely many polynomial functors $Q_1,\ldots,Q_k\lessdot P$ and polynomial transformations $\alpha_i\colon Q_i \to P$ with $X(U) \subseteq \bigcup_{i=1}^k \im(\alpha_{i,U})$ for all $U\in\Vec$. In the latter case, $X$ is contained in a proper closed subset of $P$.

If we assume furthermore that $P$ is irreducible, then in the second
case there exists a integer $k\geq 0$ such that for all $U \in \Vec$
and all $p \in X(U)$ the strength of $p$ is at most $k$.
\end{thm}

This is a strengthening of a theorem from the upcoming paper \cite{BDES:geometrypolyrep} (also appearing in the first author's thesis \cite[Theorem 4.2.5]{B:thesis}), where the additional
assumption is that $X$ be closed.

\begin{re}
When $P$ is irreducible of degree $1$, then $P(U)=U$. In this case, the subsets of $P$ are $P$ and $\{0\}$. So indeed, the elements of a proper subset of $P$ have bounded strength, namely $0$.
\end{re}

Again, the condition that $K$ be algebraically closed cannot be dropped,
and neither can the condition on the characteristic; however, see
Remark~\ref{re:Char}. Main Theorem I has the same corollary as
Theorem~\ref{thm:KZ}.

\begin{cor} \label{cor:Surjective}
With the same assumptions as in Main Theorem I, let $U\in\Vec$ be a fixed vector space. Then there exist finitely many polynomial functors $Q_1,\ldots,Q_k\lessdot P$ and polynomial transformations $\alpha_i\colon Q_i \to P$ such that for every $V \in \Vec$ and
every $f \in P(V)$ that is {\em not} in $\bigcup_{i=1}^k \im(\alpha_{i,V})$
the map $\Hom(V,U) \to P(U), \phi \mapsto P(\phi)f$ is surjective.

If $P$ is irreducible, then the condition that $f \not \in\bigcup_{i=1}^k \im(\alpha_{i,V})$
can be replaced by the condition that $f$ has strength greater than some
function of $\dim U$ only.
\end{cor}

\subsection{Limits and dense orbits}

Let $P$ be a pure polynomial functor over $K$.  There is
another point of view on closed subsets of $P$, which involves
limits that we define now.

\begin{de}
We define $P_\infty:=\varprojlim_n P(K^n)$, where the map $P(K^{n+1}) \to
P(K^n)$ is $P(\pi_n)$ with $\pi_n\colon K^{n+1} \to K^n$ the projection map forgetting
the last coordinate. We equip $P_\infty$ with the inverse limit of the
Zariski topologies on the $P(K^n)$, which is itself a Zariski topology
coming from the fact that $P_\infty=(\bigcup_n P(K^n)^*)^*$.
We also write $P(\pi_n)$ for the projection map $P_\infty \to
P(K^n)$; this will not lead to confusion. A polynomial
transformation $\alpha\colon P \to Q$ naturally yields a
continuous map $P_\infty \to Q_\infty$ also denoted by
$\alpha$.
\end{de}

If $P=S^d$, then the elements of $P_\infty$ can be thought of
as homogeneous series of degree $d$ in infinitely many variables
$x_1,x_2,\ldots$. Here, closed subsets of $P_\infty$ are defined by
polynomial equations in the coefficients of these series.

On $P_\infty$ acts the group $\GL_\infty=\bigcup_n \GL_n$,
where $\GL_n$ is embedded into $\GL_{n+1}$ via the map
\[
g \mapsto \begin{pmatrix} g & 0 \\ 0 & 1 \end{pmatrix}.
\]
Indeed, with this embedding the map $P(K^{n+1}) \to P(K^n)$
in the definition of $P_\infty$ is $\GL_n$-equivariant, and
this yields the action of $\GL_\infty$ on the projective
limit.
In the case of degree-$d$ series, an element $g\in\GL_n\subset\GL_\infty$ maps each of the first $n$ variables $x_i$ to a linear combination of
$x_1,\ldots,x_n$ and the remaining variables to themselves.

The map
that sends a closed subset $X$ of $P$ to the closed subset $X_\infty:=
\varprojlim_n X(K^n)$ of $P_\infty$ is a bijection with the collection
of closed $\GL_\infty$-stable subsets of $P_\infty$ \cite[Proposition 1.3.28]{B:thesis}. Hence closed subsets of polynomial functors can also be
studied in this infinite-dimensional setting.

\begin{ex}
On degree-$d$ forms, $\GL_\infty$ clearly has dense orbits, such as
that of
\[
f=x_1 x_2 \cdots x_d + x_{d+1} x_{d+2} \cdots x_{2d} + \ldots
\]
The reason is that this series can be specialised to any degree-$d$ form
{\em in finitely many variables} by linear variable substitutions. This
implies that the image of $\GL_\infty\! \cdot f$ in each $S^d(K^n)$ is dense.
Hence $\GL_\infty\!\cdot f$ is dense in $S^d_\infty$.
\end{ex}

For every pure polynomial functor $P$, the group $\GL_\infty$ has dense orbits on
$P_\infty$---in fact, uncountably many of them! See \cite[\S4.5.1]{B:thesis}. They have
the following interesting property.

\begin{cor}\label{cor:surjective_infty}
Suppose that $\GL_\infty\!\cdot  p$ is dense in $P_\infty$. Then for each integer $n\geq1$, the image of $\GL_\infty\!\cdot p$ in $P(K^n)$ is all of $P(K^n)$.
\end{cor}
\begin{proof}
For $V\in\Vec$, define
\[
X(V):= \left\{P(\phi)P(\pi_n)p\mid n\geq 1, \phi\in\Hom(K^n,V)\right\}\subseteq P(V),
\]
which is exactly the image of $\GL_\infty\!\cdot p$ under the
projection $P_\infty \to P(K^m)$ followed by an isomorphism $P(\phi)$,
where $\phi\colon K^m \to V$ is a linear isomorphism. We see that $X$ is a
subset of $P$. For each $V\in\Vec$, the subset $X(V)$ is dense in $P(V)$ since $\GL_\infty\!\cdot  p$ is dense in $P_\infty$. So $X=P$ by Main Theorem I.
\end{proof}

The notion of strength has an obvious generalisation.

\begin{de}
Assume that $\cha K=0$. Let $P$ be a homogeneous polynomial functor. The strength of
a tensor $p \in P_\infty$ is the minimal integer $k\geq0$ such that
\[
p=\alpha_1(q_1,r_1)+\ldots+\alpha_k(q_k,r_k)
\]
for some irreducible polynomial functors $Q_i,R_i$ whose positive degrees sum up to $d$, bilinear polynomial transformations $\alpha_i\colon Q_i \oplus R_i \to P$ and elements $q_i \in Q_{i,\infty}$ and $r_i \in R_{i,\infty}$. If no such $k$ exists, we say that $p$
has infinite strength.
\end{de}

\begin{cor}
Assume that $\cha K=0$ and that $P$ is irreducible of degree $\geq2$.  Then an element of
$P_\infty$ has infinite strength if and only if its $\GL_\infty$-orbit
is dense.
\end{cor}
\begin{proof}
If $p \in P_\infty$ has finite strength, then let $\alpha_i\colon Q_i \times
R_i \to P$ be as in the definition above and let
\[
\alpha:=\alpha_1 +\ldots + \alpha_k\colon Q:=\bigoplus_{i=1}^k (Q_i \otimes R_i) \to P
\]
be their sum, so that $p \in \im(\alpha)$. Consider the closed subset $X=\overline{\im(\alpha)}$, i.e., the closed subset defined by $X(V)=\overline{\im(\alpha_V)}$ for all $V\in\Vec$. As $\dim Q(K^n)$ is
a polynomial in $n$ of degree $<d$, while $\dim P(K^n)$ is a polynomial
in $n$ of degree $d$, we see that $X(K^n)$ is a proper subset of $P(K^n)$ for all $n\gg0$. Since $p\in X_{\infty}$, it follows that $\GL_\infty\!\cdot p$ is not dense.

Suppose, conversely, that $\GL_\infty\!\cdot p$ is not dense. Then it is
contained in $X_\infty$ for some proper closed subset $X$ of $P$.
Hence $p$ has finite strength by Main Theorem I.
\end{proof}

\begin{ex}
Let $P,Q$ be homogeneous functors of the same degree $d\geq2$ and let $p\in P_{\infty}$ be an element of infinite strength. Then $(p,0)\in P_{\infty}\oplus Q_{\infty}$ also has infinite strength, but the orbit $\GL_{\infty}\!\cdot(p,0)$ is not dense.
\end{ex}

\begin{re} \label{re:Shift}
In Section~\ref{sec:ProofII} we will use a generalisation of
notation introduced here:
for an integer $m\geq0$ we will write $P_{\infty-m}$ for the
limit $\varprojlim_n P(K^{[n]-[m]})$ over all integers $n \geq m$. This space is isomorphic to
$P_\infty$, but the indices have been shifted by $m$. On $P_{\infty -
m}$ acts the group $\GL_{\infty-m} \cong \GL_\infty$, which is the union
of $\GL(K^{[n]-[m]})$ over all $n \geq m$. We denote the image of an element $p\in P_{\infty-m}$ in $P(K^{[n]-[m]})$ by $p_{[n]-[m]}$. The inclusions $\iota_n\colon K^{[n]-[m]}\to K^n$ sending $v\mapsto (0,v)$ allow us to view $P_{\infty-m}$ as a subset of $P_{\infty}$.
\end{re}

\begin{cor}\label{cor:degen+finstrength}
Let $P$ be a homogeneous polynomial functor of degree $d\geq 2$ and $m\geq0$ an integer. Let $p\in P_{\infty-m}$ be a tensor whose $\GL_{\infty-m}$-orbit is not dense and let $q\in P_{\infty}$ be an element with finite strength. Then the $\GL_{\infty}$-orbit of $p+q\in P_{\infty}$ is also not dense.
\end{cor}
\begin{proof}
Note that $p$ is contained in the image of $\alpha\colon Q_{\infty-m} \to P_{\infty-m}$
for some polynomial transformation $\alpha\colon Q \to P$ with $Q\lessdot P$ \cite[Theorem 4.2.5]{B:thesis} and $q$ is contained in the image of $\beta\colon R_{\infty} \to P_{\infty}$ for some polynomial transformation $\beta\colon R \to P$ with $\deg(R)<d$. So since $Q\oplus R\lessdot P$ by Remark~\ref{re:degen+finstrength}, we see that $p+q$ is contained in a proper closed subset of $P$. Hence its $\GL_{\infty}$-orbit is not dense.
\end{proof}

\subsection{Linear endomorphisms}\label{ss:linendo}

Our second goal in this paper is to show that there always exist {\em
minimal} $f$ with dense orbits. This minimality relates to a monoid of
linear endomorphisms extending $\GL_\infty$, as follows.
Elements of $\GL_\infty$ are $\NN\times\NN$ matrices of the block form
\[
\begin{pmatrix} g & 0 \\ 0 & I_{\infty} \end{pmatrix}
\]
where $g \in \GL_n$ for some $n$ and $I_{\infty}$ is the infinite identity
matrix.

\begin{de}
Let $E \supset \GL_\infty$ be the monoid of $\NN \times
\NN$ matrices with the property that each {\em row} contains only finitely
many nonzero entries.
\end{de}

\begin{ex}
For every integer $i\geq 1$, let $\phi_i\in K^{n_i\times m_i}$ be a matrix. Then the block matrix
\[
\begin{pmatrix}
\phi_1\\&\phi_2\\&&\ddots
\end{pmatrix}
\]
is an element of $E$.
\end{ex}

We define an action of $E$ on $P_\infty$ as follows.
Let $p=(p_0,p_1,\ldots) \in P_\infty$ and $\phi \in E$. For each integer $i\geq0$, to compute $q_i$ in
\[
q=(q_0,q_1,\ldots) =P(\phi) p
\]
we choose $n_i\geq 0$ such that all the nonzero entries of the first $i$ rows of $\phi$ are in the first $n_i$ columns. Now, we let
$\psi_i \in K^{i\times n_i}$ be the $i \times n_i$ block
in the upper-left corner of $\phi$, so that
\[
\phi=\begin{pmatrix} \psi_i & 0 \\ * & * \end{pmatrix},
\]
and we set $q_i:=P(\psi_i)p_{n_i}$. Note that if we replace $n_i$
by a larger number $\tilde{n}_i$, then the resulting matrix
$\tilde{\psi}_i$ satisfies $\tilde{\psi}_i = \psi_i \circ
\pi$, where $\pi:K^{\tilde{n}_i} \to K^{n_i}$ is the
projection. Consequently, we then have
\[
P(\tilde{\psi}_i) p_{\tilde{n}_i}=P(\psi_i)P(\pi)p_{\tilde{n}_i} = P(\psi_i) p_{n_i},
\]
so that $q_i$ is, indeed, well-defined. A straightforward
computation shows that, for $\phi,\psi \in E$, we have
$P(\psi) \circ P(\phi)=P(\psi \circ \phi)$, so that $E$ does
indeed act on $P_\infty$.

For infinite degree-$d$ forms, the action of $\phi \in E$ is by linear
variable substitutions $x_j \mapsto \sum_{i=1}^{\infty} \phi_{ij} x_i$. Note
that, since each $x_i$ appears in the image of only finitely many $x_j$,
this substitution does indeed make sense on infinite
degree-$d$ series.

Since $\GL_\infty \subseteq E$, an $E$-stable subset of $P_\infty$
is also $\GL_\infty$-stable.  The converse does not hold, since for
instance $E$ also contains the zero matrix, and $P(0)f=0 \neq P(g)f$
for all nonzero $f\in P_{\infty}$ and $g \in \GL_\infty$ when the polynomial functor $P$ is pure. However, it is easy to see
that $\GL_\infty$-stable {\em closed} subsets of $P_\infty$ are also
$E$-stable. In particular, we have $\overline{\GL_{\infty}\!\cdot f}=\overline{P(E)f}$.

\subsection{A quasi-order on infinite tensors}\label{sec:quasiorder}

\begin{de}
For infinite tensors $p,q \in P_\infty$ we write $p \preceq q$ if $p \in P(E)q$. In this case, we say that $q$ {\em specialises} to $p$.
\end{de}

From the fact that $E$ is a unital monoid that acts on $P_\infty$,
we find that $\preceq$ is transitive and reflexive. Hence it induces an
equivalence relation $\simeq$ on $P_\infty$ by
\[
p \simeq q :\Leftrightarrow p \preceq q \text{ and } q \preceq p,
\]
as well as a partial order on the equivalence classes of $\simeq$.

\begin{ex}\label{ex:copiesS}
Fix an integer $k\geq1$ and consider the polynomial functor $P = (S^1)^{\oplus k}$. A tuple $q=(q_1,\ldots,q_k)\in P_{\infty}$ has a dense $\GL_{\infty}$-orbit if and only if $q_1,\ldots,q_k\in S^1_{\infty}$ are linearly independent. Suppose that $q$ has a dense $\GL_{\infty}$-orbit and let $A$ be the $\NN\times k$ matrix corresponding to $q$. Then $A$ has full rank. By acting with an element of $\GL_{\infty}\subseteq E$, we may assume that
\[
A=\binom{I_k}{B}
\]
where $B$ is again an $\NN\times k$ matrix. Now, take
\[
\phi_C:=\begin{pmatrix}
I_k\\C&I_{\infty}
\end{pmatrix}\in E
\]
and note that $\phi_{-B}A=(I_k~0)^\top$, so that $P(\phi_{-B})q=(x_1,\ldots,x_k)$. So any two tuples in $P_\infty$ with a dense $\GL_\infty$-orbit are in the same equivalence class. Moreover, the element of $E$ specialising one tuple to the other can be chosen to be invertible in $E$ as $\phi_{C}\varphi_{-C}=I_{\infty}$.
\end{ex}

There is an obvious relation between $\preceq$ and orbit closures, namely:
if $p \preceq q$, then $p \in \overline{\GL_\infty\!\cdot q}$. The converse,
however, is not true.

\begin{ex}
Let $p=x_1(x_1^2+x_2^2+\ldots),q=x_1^3+x_2^3+\ldots\in S^3_{\infty}$. Then $q$ has infinite strength and so $p\in S^3_{\infty}=\overline{\GL_{\infty}\!\cdot q}$. However, we have $p\not\preceq q$: suppose that
\[
f := x_1g(x_1,x_2,\ldots)+h(x_2,x_3,\ldots)\in S^3(E)q
\]
for some $g\in S^2_\infty$ and $h\in S^3_\infty$. As only finitely many variables $x_i$ are substituted by linear forms containing $x_1$ when specialising $q$ to $f$, we see that
\[
x_1g(x_1,x_2,\ldots)+\tilde{h}(x_2,x_3,\ldots)\in S^3(E)(x_1^3+x_2^3+\ldots+x_n^3)
\]
for some integer $n\geq 1$ and $\tilde{h}\in S^3_{\infty}$. From this, it is easy to see that $g$ has finite strength. Hence $f\neq p$ as $x_1^2+x_2^2+\ldots$ has infinite strength. So indeed $p\not\preceq q$.
\end{ex}

In order to have a tensor $p\in P_{\infty}$ with a dense $\GL_{\infty}$-orbit, the polynomial functor $P$ must be pure. For some time, we believed that when this is the case all elements $p \in P_\infty$ with a dense $\GL_\infty$-orbit might form a single $\simeq$-equivalence class. When $P$ has degree $\leq2$, this is in fact true; see Example~\ref{ex:degreeleq2}. However, it doesn't hold for cubics.

\begin{ex}\label{ex:noncomparable}
Let $p,q\in S^3_{\infty}$ be as before. Now also consider $r = p(x_1,x_3,\ldots)+q(x_2,x_4,\ldots)$. We have $q=r(0,x_1,0,x_2,\ldots)\preceq r$ and so $S^3_{\infty}=\overline{\GL_{\infty}\!\cdot q}\subseteq\overline{\GL_{\infty}\!\cdot r}$. Hence both $q$ and $r$ have dense $\GL_{\infty}$-orbits. And, we have $r\not\preceq q$: indeed, otherwise $p=r(x_1,0,x_2,0,\ldots)\preceq r\preceq q$, but $p\not\preceq q$.
\end{ex}

\subsection{Minimal classes of elements with dense orbits}

Our second main result is the following.

\begin{thm}[Main Theorem II] \label{thm:MainII}
Suppose that $K$ is algebraically closed of characteristic zero. Let
$P$ be a pure homogeneous polynomial functor over $K$. Then
there exist tensors $p,r \in P_\infty$ whose
$\GL_\infty$-orbits are dense such that $p\preceq q\preceq r$ for all $q\in P_{\infty}$ whose $\GL_\infty$-orbit is dense.
\end{thm}

The elements $p$ that have this property form a single $\simeq$-class
which lies below the $\simeq$-classes of all other $q\in P_{\infty}$
whose $\GL_\infty$-orbit is dense.  For the construction of such a tensor
$p\in P_{\infty}$, see \S\ref{ssec:minimal_class}. For the construction
of the tensor $r \in P_\infty$, see \S\ref{ssec:maximal_class}.

\begin{re} \label{re:Char}
In both our Main Theorems, we require that the characteristic
be zero. This is because the results in \cite{B:thesis} and
\cite{BDES:geometrypolyrep} require this. However, the proof of
topological Noetherianity for polynomial functors in \cite{D:topnoeth}
does not require characteristic zero, and shows that after a shift
and a localisation, a closed subset of a polynomial functor admits a
homeomorphism into an open subset of a smaller polynomial functor. In
characteristic zero, this is in fact a closed embedding, so that it
can be inverted and yields a parameterisation of (part of) the closed
subset. In positive characteristic, it is not a closed embedding, but
the map still becomes invertible if one formally inverts the Frobenius
morphism; this is touched upon in \cite{BDES:geometrypolyrep}. This
might imply variants of our Main Theorems in arbitrary characteristic,
but we have not yet pursued this direction in detail.
\end{re}

\section{Proof of Main Theorem I} \label{sec:ProofI}

\subsection{The linear approximation of a polynomial functor}
\label{ssec:Derivative}

Let $P$ be a polynomial functor over an infinite field and let $U,V
\in \Vec$. Then $P(U \oplus V) = \bigoplus_{d,e=0}^{\infty}
Q_{d,e}(U,V)$ where
\[
Q_{d,e}(U,V):=\{v \in P(U \oplus V) \mid \forall s,t \in K: \ P(s\id_U \oplus\, t\id_V)v = s^d t^e v \}\subseteq P_{d+e}(U\oplus V).
\]
The terms with $e=0$ add up to $P(U)$, and the terms with $e=1$ add up
to a polynomial bifunctor evaluated at $(U,V)$ that is linear in $V$.
This is necessarily of the form $P'(U) \otimes V$, where $P'$ is a
polynomial functor. In other words, we have
\[ P(U \oplus V)=P(U) \oplus (P'(U) \otimes V) \oplus \text{
higher-degree terms in }V. \]
We informally think of the first two terms as the linear approximation
of $P$ around $U$.  Now suppose that we have a
short exact sequence
\[ 0 \to P \to Q \to R \to 0 \]
of polynomial functors. This implies that for all $U,V$ we have a short
exact sequence
\[ \{0\} \to P(U \oplus V) \to Q(U \oplus V) \to R(U \oplus V) \to \{0\} \]
and inspecting the degree-$1$ parts in $V$ we find a short
exact sequence
\[ 0 \to P' \to Q' \to R' \to 0. \]
This, and further straightforward computations, shows that $P \mapsto P'$
is an exact functor from the category of polynomial functors to itself.

\begin{re}
For $U\in\Vec$ fixed, denote the polynomial functor sending $V\mapsto P(U\oplus V)$ and $\phi\mapsto P(\id_U\oplus \phi)$ by $\Sh_U(P)$. Then we have
\[
\Sh_U(P)_e(V)=\{v \in P(U \oplus V) \mid \forall t \in K: \ P(\id_U \oplus\, t\id_V)v = t^e v \}
\]
and from this we see that $Q_{d,e}(U,V)=\Sh_U(P)_e(V)\cap P_{d+e}(U\oplus V)$. In particular, when $P$ is homogeneous of degree $d$, we see that $P(U\oplus V)=\bigoplus_{e=0}^d Q_{d-e,e}(U,V)$ where $Q_{d-e,e}(U,V)=\Sh_U(P)_e(V)$. Also note that, in this case, $\Sh_U(P)_0(V)=P(U)$ and $\Sh_U(P)_d(V)=P(V)$ via the inclusions of $U,V$ into $U\oplus V$.
\end{re}

\begin{ex}
If $P=S^d$, then the formula
\[
S^d(U \oplus V) \cong \bigoplus_{e=0}^d S^{d-e}(U) \otimes S^e(V)
=S^d (U) \oplus (S^{d-1} (U) \otimes V) \oplus\cdots
\]
identifies $P'$ with $S^{d-1}$.
\end{ex}

\begin{ex} \label{ex:CharP}
Let $K$ be an algebraically closed field of characteristic $p$. Then
$S^p$ contains the subfunctor $P(V):=\{v^p \mid v \in V\}$.  We have
$P(U \oplus V)=P(U) \oplus P(V)$, and hence $P'=0$.
\end{ex}

\subsection{Proof of Main Theorem I}

In this subsection we prove Theorem~\ref{thm:MainI}.  We start with a
result of independent interest.

\begin{thm} \label{thm:DenseImpliesEqual}
Let $P$ be a pure polynomial functor over an algebraically
closed field $K$of characteristic $0$ or $> \deg(P)$
and let $X$ be a subset of $P$ such that $X(V)$ is dense in $P(V)$
for all $V \in \Vec$. Then, in fact, $X(V)$ is {\em equal} to $P(V)$
for all $V \in \Vec$.
\end{thm}

Example~\ref{ex:PosDef} shows that the condition that $K$ be algebraically
closed cannot be dropped. We do not know if the condition on the
characteristic of $K$ can be dropped, but the proof will use that
the polynomial functor $P'$ introduced in \S\ref{ssec:Derivative} is
sufficiently large, which, by Example~\ref{ex:CharP}, need not be the
case when $\cha K$ is too small.

\begin{proof}
Let $q \in P(K^n)$. For each $k \geq n$, we consider the incidence variety
\[
Z_k:=\{(\phi,r) \in \Hom(K^k,K^n) \times P(K^k)
\mid \rk(\phi)=n \text{ and } P(\phi)r=q\}.
\]
We write $e_k:=\dim_K P(K^k)$. Since for every $\phi \in
\Hom(K^k,K^n)$ of rank $n$ the linear map $P(\phi)$ is
surjective, $Z_k$ is a vector
bundle of rank $e_k-e_n$ over the rank-$n$ locus in $\Hom(K^k,K^n)$.
Hence $Z_k$ is an irreducible variety with $\dim Z_k=kn + e_k-e_n$. We
therefore expect the projection $\Pi\colon Z_k \to P(K^k)$ to be dominant for
$k \gg n$. To prove that this is indeed the case, we need to show that for
$z \in Z_k$ sufficiently general, the local dimension at $z$ of the fibre
$\Pi^{-1}(\Pi(z))$ is (at most) $\dim(Z_k)-e_k=kn-e_n$. By the upper semicontinuity
of the fibre dimension \cite[Theorem 11.12]{H:ag}, it suffices to exhibit a single point $z$ with this
property, and indeed, it suffices to show that the tangent space to the
fibre at $z$ has dimension (at most) $kn-e_n$.

To find such a point $z$, set $U:=K^n$ and $V:=K^{k-n}$ and consider
\[
z:=(\pi_U,P(\iota_U)q + r) \in Z_k,
\]
where $\pi_U\colon U \oplus V \to U$ is the projection and $\iota_U\colon U \to
U \oplus V$ is the inclusion and where we will choose $r \in P'(U)
\otimes V \subseteq P(U \oplus V)$. Note that then
\[
P(\iota_U)q + r \in P(U) \oplus (P'(U) \otimes V) \subseteq
P(U \oplus V)
\]
and that $P(\pi_U)r=0$ so that $z$ does, indeed, lie in $Z_k$.

The tangent space $T_z \Pi^{-1}(\Pi(z))$ (projected into $\Hom(K^k,K^n)$)
is contained in the solution space of the linear system of equations
\[
P(\pi_U + \epsilon \psi)(P(\iota_U) q + r) = q \mod \epsilon^2
\]
for $\psi$. By the rank theorem, the dimension of this
solution space equals $kn=\dim(\Hom(K^k,K^n))$ minus the rank of the linear map
\[
\Hom(U \oplus V,U) \to P(U), \psi \mapsto \text{ the
coefficient of $\epsilon$ in }
P(\pi_U + \epsilon \psi)(P(\iota_U) q + r).
\]
So it suffices to prove that for all $k \gg n$ there is a suitable $r$ such that this linear map is surjective. In fact, we will restrict the domain to those $\psi
\in \Hom(U \oplus V,U)$ of the form $\omega\circ \pi_V$ where $\pi_V\colon U \oplus V \to V$ is the projection and $\omega\in\Hom(V,U)$. Then
\[
P(\pi_U + \epsilon \psi)(P(\iota_U) q) =P((\pi_U + \epsilon \omega\circ \pi_V)\circ\iota_U) q=P(\id_U)q=q
\]
So $P(\iota_U)q$ does not contribute to the coefficient of $\epsilon$ and this coefficient equals
\[
P(\id_U + \id_U) (\id_{P'(U)} \otimes \,\omega) r
\]
where $\id_U + \id_U\colon U \oplus U \to U$ is the map sending $(u_1,u_2)$
to $u_1 + u_2$. Note that the codomain of $\id_{P'(U)} \otimes \,\omega$ equals
$P'(U) \otimes U \subseteq P(U \oplus U)$, so that the composition above
makes sense. Below we will show that for $k-n = \dim V \gg n$ and suitable
$r \in P'(U) \otimes V$ the linear map
\begin{align*}
\Omega_{P,V,r}\colon\Hom(V,U) &\to P(U)\\
\omega &\mapsto P(\id_U + \id_U)(\id_{P'(U)} \otimes \,\omega)r
\end{align*}
is surjective.

Hence there exists a $k$ such that $Z_k \to P(K^k)$ is dominant. By
Chevalley's theorem, the image contains a dense open subset of $P(K^k)$,
and this dense open subset intersects the dense set $X(K^k)$. Hence
there exists an element $p \in X(K^k)$ and a $\phi \in \Hom(K^k, K^n)$
such that $P(\phi)p=q$. Finally, since $X$ is a subset of $P$, also $q$
is a point in $X(K^n)$. Hence $X(K^n)=P(K^n)$ for each $n$, as desired.
\end{proof}

\begin{lm}
Let $P$ be a polynomial functor over an infinite field $K$ with
$\cha(K)=0$ or $\cha(K) > \deg(P)$ and let $U \in \Vec$. Then for $V \in \Vec$
with $\dim V \gg \dim U$, there exists an $r \in P'(U) \otimes V$ such that
\begin{align*}
\Omega_{P,V,r}\colon\Hom(V,U) &\to P(U)\\
\omega &\mapsto P(\id_U + \id_U)(\id_{P'(U)} \otimes \,\omega)r
\end{align*}
is surjective.
\end{lm}
\begin{proof}
When $\cha(K)=0$, the Abelian category of polynomial functors is semisimple, with the Schur functors as a basis. When $\cha(K)=p>0$, the situation is more complicated. The irreducible polynomial functors still correspond to partitions \cite[Theorem 3.5]{G:polyreps}. A degree-$d$ irreducible polynomial functor is a submodule of the functor $T(V)=V^{\otimes d}$ if and only if the corresponding partition is column $p$-regular \cite[Theorem 3.2]{J:decompprime}. %\cite[Corllary 6.4g]{G:polyreps}.
 Luckily, this is always the case when $d<p$. And, the Abelian category of polynomial functors of degree $<p$ is semisimple~\cite[Corollary 2.6e]{G:polyreps}. Now, if
$P,Q$ are such polynomial functors and $r_1 \in P'(U) \otimes
V$ and $r_2 \in Q'(U) \otimes W$ have the required property for
$P,Q$, respectively, then
\begin{align*}
r:=(r_1,r_2) \in (P'(U) \otimes V) \oplus (Q'(U)
\otimes W) &\subseteq (P'(U) \oplus Q'(U)) \otimes (V
\oplus W)\\ & = (P\oplus Q)'(U) \otimes (V \oplus W)
\end{align*}
has the required property for $P\oplus Q$. Hence it suffices to prove the lemma
in the case where $P$ is an irreducible polynomial functor of degree $d$.
We then have $T=P \oplus Q$, where $T(V)=V^{\otimes d}$ and $Q$ is another
polynomial functor. By a similar argument as above, if $r \in T'(U)
\otimes V$ has the required property for $T$, then its image in $P'(U)
\otimes V$ has the required property for $P$. Hence it suffices to prove
the lemma for $T$.

Now  we have
\begin{align*} T(U \oplus V)=T(U) &\oplus
(V \otimes U \otimes U \otimes \cdots \otimes U)  \oplus
(U \otimes V \otimes U \otimes \cdots \otimes U) \\ &  \oplus
\cdots
\oplus
(U \otimes U \otimes U \otimes \cdots \otimes V) \oplus
\text{ terms of higher degree in $V$,}
\end{align*}
so that $T'$ is a direct sum of $d$ copies of $U \mapsto
U^{\otimes d-1}$. We
take $r$ in the first of these copies, as follows. Let $e_1,\ldots,e_n$
be a basis of $U$ and set
\[ r:=\sum_{\alpha \in [n]^{d-1}}
v_{\alpha} \otimes e_{\alpha_1} \otimes \cdots \otimes
e_{\alpha_{d-1}} \]
where the $v_\alpha$ are a basis of a space $V$ of dimension
$n^{d-1}$. For every $\beta \in [n]^{d-1}$ and $i \in [n]$, the linear map
$\omega$ that maps $v_{\beta}$ to $e_i$ and all other
$v_\alpha$ to zero is a witness to the fact that $e_i
\otimes e_{\beta_1} \otimes \cdots \otimes e_{\beta_{d-1}}$
is in the image of $\Omega_{T,V,r}$. Hence this linear map
is surjective.
\end{proof}

\begin{lm}
Assume that $K$ is algebraically closed of characteristic zero. Let $P,Q$ be polynomial functors. Assume that $P$ is irreducible of degree $d$, $Q$ has degree $<d$ and let $\alpha\colon Q \to P$ be a polynomial
transformation, then there is a uniform bound on the strength of elements
of $\im(\alpha_V)$ that is independent of $V$.
\end{lm}
\begin{proof}
Let $R$ be the sum of the components of $Q$ of strictly positive degree. Any
element in $\im(\alpha_V)$ is also in $\im(\beta_{V})$ for a
polynomial transformation $\beta_V\colon R\to P$ obtained
from $\alpha$ by a suitable specialisation. Write $R=R^{(1)} \oplus\cdots \oplus R^{(k)}$, where the $R^{(i)}$ are Schur functors of degrees $0<d_i<d$. The
polynomial transformation $\beta$ factors uniquely as the polynomial transformation
\begin{align*}
\delta\colon R^{(1)} \oplus\cdots \oplus R^{(k)}&\to F:=\bigoplus_{\substack{e_1,\ldots,e_k\geq0\\\sum_i e_i d_i=d}} \bigotimes_{i=1}^k S^{e_i} R^{(i)} \\
(r_1,\ldots,r_k)&\mapsto(r_1^{\otimes e_1}\otimes\cdots\otimes r_k^{\otimes e_k})_{e_1,\ldots,e_k}
\end{align*}
and a {\em linear} polynomial transformation $\gamma\colon F \to P$. As $\gamma$ is linear, we see that $\str(\gamma_V(v))\leq \str(v)$ for all $V\in\Vec$ and $v\in F(V)$. So it suffices to prove that the elements of the subset $\im(\delta)$, which depends only on $Q$ and $d$, have bounded strength. We have
\[
\str(r_1^{\otimes e_1}\otimes\cdots\otimes r_k^{\otimes e_k})_{e_1,\ldots,e_k}\leq \sum_{\substack{e_1,\ldots,e_k\geq0\\\sum_i e_i d_i=d}}\str(r_1^{\otimes e_1}\otimes\cdots\otimes r_k^{\otimes e_k})\leq\sum_{\substack{e_1,\ldots,e_k\geq0\\\sum_i e_i d_i=d}}1
\]
as $\sum_i e_i\geq 2$ whenever $\sum_i e_id_i=d$. So this is indeed the case.
\end{proof}

\begin{proof}[Proof of Theorem~\ref{thm:MainI} (Main Theorem I)]
Let $X$ be a subset of a pure polynomial functor $P$ over an algebraically
closed field $K$ of characteristic zero. For each
$V \in \Vec$ define $Y(V):=\overline{X(V)}$. If $Y$ is a proper closed
subset of $P$, then by \cite[Theorem 4.2.5]{B:thesis} there exist finitely many polynomial
transformations $\alpha_i\colon Q_i \to P$ with $Q_i\lessdot P$
and $Y(V) \subseteq \bigcup_i \im(\alpha_{i,V})$ for all $V\in\Vec$.
Since $X \subseteq Y$, we are done. Otherwise, if $Y(V)=P(V)$
for all $V$, then Theorem~\ref{thm:DenseImpliesEqual} implies that also
$X(V)=P(V)$ for all $V$. The last statement follows from the previous lemma.
\end{proof}

\begin{proof}[Proof of Corollary \ref{cor:Surjective}]
Let $X$ be the subset of $P$ consisting of all elements $f\in P(V)$ such that
\begin{align*}
\Hom(V,U) &\to P(U)\\
\phi &\mapsto P(\phi)f
\end{align*}
is not surjective. By Main Theorem I, it suffices to prove that $X\neq P$. As before, we claim that in fact $X(V)\neq P(V)$ already when $\dim V\geq \deg(P)\cdot\dim P(U)$.

First suppose that $P$ is irreducible. Then $P$ is a Schur functor. Take $V_0=K^d$ and $\ell=\dim P(U)$. Then it is known that $\Hom(V_0,U)\cdot P(V_0)$ spans $P(U)$. Let $P(\phi_1)p_1,\ldots,P(\phi_\ell)p_{\ell}$ be a basis of $P(U)$, let $\iota_i\colon V_0\to V_0^{\oplus \ell}$ and $\pi_i\colon V_0^{\ell}\to V_0$ be the inclusion and projection maps and take
\[
p=P(\iota_i)p_1+\ldots+P(\iota_\ell)p_\ell\in P(V_0^{\oplus \ell}).
\]
Then $P(\phi_i\circ\pi_i)(p)=P(\phi_i)p_i$. Hence
\begin{align*}
\Hom(V_0^{\oplus\ell},U) &\to P(U)\\
\phi &\mapsto P(\phi)p
\end{align*}
is surjective.

Next, suppose that $P=Q\oplus R$ and that there exist $f\in Q(V)$ and $g\in R(W)$ such that
\begin{align*}
\Hom(V,U) &\to Q(U)&\mbox{and}&&\Hom(W,U) &\to R(U)\\
\phi &\mapsto Q(\phi)f&&&\phi&\mapsto R(\phi)g
\end{align*}
are surjective. By induction, we can assume such $f,g$ exist when $\dim V \geq \deg(P)\cdot \dim Q(U)$ and $\dim W\geq\deg(P)\cdot \dim R(U)$. Now, we see that
\begin{align*}
\Hom(V\oplus W,U) &\to P(U)\\
\phi &\mapsto P(\phi)(P(\iota_1)(f)+P(\iota_2)(g))
\end{align*}
is surjective. This proves the first part of the corollary. For the second statement, we note that when $P$ is irreducible the elements of $\im(\alpha_i)$ have bounded strength. As the bound depends only on $X$ and $X$ only depends on $\dim U$, we see that $f \not \in\bigcup_{i=1}^k \im(\alpha_i)$ for all $f$ with strength greater than some function of $\dim U$ only.
\end{proof}

\section{Proof of Main Theorem II} \label{sec:ProofII}

\subsection{Construction of the minimal class}\label{ssec:minimal_class}
Let $P$ be a homogeneous polynomial functor of degree $d>0$ over an
algebraically closed field $K$ of characteristic zero. Decompose
\[
P=P^{(1)} \oplus \cdots \oplus P^{(\ell)}
\]
into Schur functors. For each $U \in \Vec$ of dimension $\geq d$ the
$\GL(U)$-module $P^{(i)}(U)$ is irreducible (and in
particular nonzero).  Let $V \in \Vec$ be a
vector space of dimension $d$. Let $V^{(1,i)}$ be a copy of~$V$ for each $i=1,\ldots,\ell$ and choose any nonzero $q^{(1,i)} \in
P^{(i)}(V^{(1,i)})$. We write
\[
q^{(1)}:=q^{(1,1)} + \ldots + q^{(1,\ell)} \in
P^{(1)}(V^{(1,1)}) \oplus \cdots \oplus P^{(\ell)}(V^{(1,\ell)})
\subseteq P(W^{(1)})
\]
where $W^{(1)}=V^{(1,1)} \oplus \cdots \oplus V^{(1,\ell)}$.
We take independent copies $W^{(j)}=V^{(j,1)} \oplus \cdots \oplus V^{(j,\ell)}$ of $W^{(1)}$ and copies $q^{(j)}=q^{(j,1)} + \ldots + q^{(j,\ell)} \in P(W^{(j)})$ of $q_1$ and set
\[
q:=q^{(1)} + q^{(2)} + \ldots \in P_\infty
\]
where we concatenate copies of a basis in the $\ell d$-dimensional space $W^{(1)}$
to identify $W^{(1)} \oplus \cdots \oplus W^{(k)}$ with~$K^{k \ell d}$.

\begin{ex}
Let $P=S^d \oplus \Wedge^d$, so that we may take $V=K^d$. We may
take $q^{(1,1)}:=x_1^d \in S^d(V^{(1,1)})$ and $q^{(1,2)}:=x_{d+1} \wedge
\cdots \wedge x_{2d} \in \Wedge^d(V^{(1,2)})$, where $x_1,\ldots,x_d$
and $x_{d+1},\ldots,x_{2d}$ are bases of $V^{(1,1)}$ and $V^{(1,2)}$,
respectively. We then have
\[
q=(x_1^d + x_{d+1}\wedge \cdots \wedge x_{2d})
+ (x_{2d+1}^d + x_{3d+1}\wedge \cdots \wedge x_{4d}) + \ldots
\qedhere
\]
\end{ex}

We will prove, first, that any $q$ constructed in this manner has a
dense $\GL_\infty$-orbit in $P_\infty$, and second, that $q \preceq p$
for all $p \in P_\infty$ with a dense $\GL_\infty$-orbit.

\subsection{Density of the orbit of $q$}

\begin{prop} \label{prop:Dense}
The $\GL_\infty$-orbit of $q$ is dense in $P_\infty$.
\end{prop}
\begin{proof}
It suffices to prove that for each $U \in \Vec$ and each $p \in P(U)$
there exists a $k\geq 1$ and a linear map $\phi\colon W^{(1)} \oplus \cdots \oplus W^{(k)} \to U$ such that $P(\phi)(q^{(1)} + \ldots + q^{(k)})=p$.  Furthermore, we may
assume that $U$ has dimension at least $d$. Fix a linear
injection $\iota\colon V \to U$. Now $\tilde{q}^{(i)}:=P(\iota)(q^{(j,i)})$ is a nonzero vector in
the $\GL(U)$-module $P^{(i)}(U)$, which
is irreducible. Hence the component $p^{(i)}$ of $p$ in $P^{(i)}(U)$
can be written as
\[
p^{(i)}=P(g^{(1,i)})\tilde{q}^{(i)} + \ldots + P(g^{(k_i,i)})\tilde{q}^{(i)}
\]
for suitable elements $g^{(1,i)},\ldots,g^{(k_i,i)} \in \End(U)$. Do
this for all $i=1,\ldots,\ell$. By taking the maximum of the numbers
$k_i$ (and setting the irrelevant $g^{(j,i)}$ equal to zero) we may
assume that the $k_i$ are all equal to a fixed number
$k$; this is the $k$ that we needed. Now
we may define $\phi$
%on the $j$-th copy $W_j$ of $W_1=V^{(1)} \oplus \cdots \oplus V^{(\ell)}$
by declaring its restriction on $V^{(j,i)}$ to be equal to $g^{(j,i)} \circ
\iota$. We then have
\[
P(\phi)(q_1+\ldots+q_k)
= \sum_{j=1}^k \sum_{i=1}^\ell P(g^{(j,i)}) \tilde{q}^{(i)}
= \sum_{i=1}^\ell p^{(i)} = p,
\]
as desired.
\end{proof}

\subsection{Minimality of the class of $q$}

\begin{prop} \label{prop:Specialising}
We have $q \preceq p$ for every $p \in P_\infty$ with a dense
$\GL_\infty$-orbit.
\end{prop}
\begin{proof}
Let $p \in P_\infty$ be a tensor with a dense
$\GL_\infty$-orbit and write $p=(p_0,p_1,p_2,\ldots)$ with $p_i \in P(K^{i})$. Take $m_0=n_0=0$. There exists a linear map $\phi_0\colon K^{m_0}\to K^{n_0}$ such that $P(\phi_0)p_{m_0}=q_{n_0}=0$, namely the zero map. Write $n_i=n_0+i\ell d$. Our goal is to construct, for each integer $i\geq 1$, an integer $m_i\geq m_{i-1}$ and a linear map $\psi_i\colon K^{[m_i]-[m_{i-1}]}\to W^{(i)}$ such that the linear map $\phi_i\colon K^{m_i}\to K^{n_i}$ making the diagram
\[ \xymatrix{
K^{m_i}=K^{m_{i-1}}\oplus K^{[m_i]-[m_{i-1}]}\ar[rr]^{\phi_i} \ar[dr]_{\id_{m_{i-1}}\oplus\,\psi_i\phantom{hello}} && K^{n_{i-1}}\oplus W^{(i)}=K^{n_i}\\
& K^{m_{i-1}}\oplus W^{(i)}\ar[ru]_{\phantom{world}\phi_{i-1}\oplus\,\id_{W^{(i)}}}
} \]
commute satisfies $P(\phi_i)p_{m_i}=q_{n_i}=q^{(1)}+\ldots+q^{(i)}$.

Let $i\geq 1$ be an integer. As observed in \S\ref{ssec:Derivative}, we can write
\[
P(K^{m_{i-1}} \oplus V) = P(K^{m_{i-1}})\oplus R_1(V)\oplus\cdots\oplus R_{d-1}(V)\oplus P(V)
\]
where $R_j=\Sh_{K^{m_{i-1}}}(P)_j$ is a homogeneous polynomial functor of degree $j$.
Writing $K^\NN$ as $K^{m_{i-1}} \oplus K^{\NN-[m_{i-1}]}$, we obtain a
corresponding decomposition
\[
p=p_{m_{i-1}} + r_1 + \ldots + r_{d-1} + p'
\]
where $r_j\in R_{j,\infty-m_{i-1}}$ and $p'\in P_{\infty-m_{i-1}}$ and we claim that $p'$ has a dense $\GL_{\infty-m_{i-1}}$-orbit; here we use
the notation from Remark~\ref{re:Shift}.

The polynomial bifunctor $(U,V)\mapsto P(U\oplus V)$ is a direct sum of bifunctors of the form $(U,V)\mapsto Q(U)\otimes R(V)$ where $Q,R$ are Schur functors. It follows that $R_j(V)$ is the direct sum of spaces $Q(K^{m_{i-1}})\otimes R(V)$ where $Q,R$ are Schur functors of degrees $d-j,j$, respectively. Hence the elements $r_1,\ldots,r_{d-1}$ have finite strength. Also note that $p_{m_{i-1}}\in P(K^{m_{i-1}})$ has finite strength.
So by Corollary~\ref{cor:degen+finstrength}, we see that the $\GL_{\infty-m_{i-1}}$-orbit of $p'$ must be dense.

The tuple $(r_1,\ldots,r_{d-1}) \in \bigoplus_{j=1}^{d-1}R_{j,\infty-m_{i-1}}$ may not have a dense $\GL_{\infty-m_{i-1}}$-orbit. However,
there exists a polynomial functor $R$ less than or equal to $R_1 \oplus
\cdots \oplus R_{d-1}$ with $R(\{0\})=\{0\}$, an $r \in R_{\infty-m_{i-1}}$ and
a polynomial transformation
\[
\alpha=(\alpha_1,\ldots,\alpha_{d-1})\colon R \to R_1 \oplus \cdots \oplus R_{d-1}
\]
such that $r$ has a dense $\GL_{\infty-m_{i-1}}$-orbit and
$\alpha(r)=(r_1,\ldots,r_{d-1})$. Since $P$ is homogeneous of degree $d>\deg(R)$, the pair $(r,p')$ has a dense orbit in $R_{\infty-m_{i-1}} \oplus P_{\infty-m_{i-1}}$ by \cite[Lemma 4.5.3]{B:thesis}. Hence, by Corollary~\ref{cor:surjective_infty}, there exists
an $m_i \geq m_{i-1} + \ell d$ and a linear map $\psi_i\colon
K^{[m_i]-[m_{i-1}]} \to W^{(i)}$ such that
$R(\psi_i)r_{[m_i]-[m_{i-1}]} = 0$ and $P(\psi_i)p'_{[m_i]-[m_{i-1}]}=q^{(i)}$.

Since polynomial transformations between polynomial functors
with zero constant term map zero to zero, the first equality
implies that, for all $j=1,\ldots,d-1$,
\[
R_j(\psi_i)r_{j,[m_i]-[m_{i-1}]}=R_j(\psi_i)\alpha_j(r_{[m_i]-[m_{i-1}]})
= \alpha_j(R(\psi_i)r_{[m_i]-[m_{i-1}]}) = \alpha_j(0)=0.
\]
Thus, informally, applying the map $\psi_i$ makes $p'$
specialise to the required $q^{(i)}$, while the terms
$r_1,\ldots,r_{d-1}$ are specialised to zero.

We define $\phi_i$ as above and we have
\begin{align*}
P(\phi_i)p_{m_i}&=P(\phi_{i-1}\oplus\id_{W^{(i)}}) P(\id_{m_{i-1}}\oplus\,\psi_i)\left(p_{m_{i-1}} + \sum_{j=1}^{d-1}r_{j,[m_i]-[m_{i-1}]} + p'_{[m_i]-[m_{i-1}]}\right)\\
&=P(\phi_{i-1}\oplus\id_{W^{(i)}}) \left(p_{m_{i-1}} + \sum_{j=1}^{d-1}R_j(\psi_i)r_{j,[m_i]-[m_{i-1}]}+ P(\phi_i)p'_{[m_i]-[m_{i-1}]}\right)\\
&=P(\phi_{i-1}\oplus\id_{W^{(i)}})(p_{m_{i-1}}+q^{(i)})=q_{n_{i-1}}+q^{(i)}=q^{(1)}+\ldots+q^{(i)}.
\end{align*}
Iterating this argument, we find that the infinite matrix
\[
\begin{pmatrix}\phi_0\\&\psi_1\\&&\psi_2\\&&&\psi_3\\&&&&\ddots\end{pmatrix}=: e
\]
has the property that $P(e)p=q^{(1)}+q^{(2)}+\ldots=q$, as desired.
\end{proof}

\begin{re}
Note that the element $e \in E$ constructed above has only finitely many
nonzero entries in each row {\em and} in each column!
\end{re}

\begin{re}\label{re:addingS1s}
Fix an integer $k\geq0$. Then we have the following strengthening of the previous theorem: we have $(x_1,\ldots,x_k,q)\preceq (\ell_1,\ldots,\ell_k,p)$ for every $(\ell_1,\ldots,\ell_k,p)\in (S^1_{\infty})^{\oplus k}\oplus P_{\infty}$ with a dense $\GL_{\infty}$-orbit. Here $q$ is defined as before in variables distinct from $x_1,\ldots,x_k$. To see this, note that a tensor in $(S^1_{\infty})^{\oplus k}\oplus P_{\infty}$ with a dense $\GL_{\infty}$-orbit is of the form $(\ell_1,\ldots,\ell_k,p)$ where $\ell_1,\ldots,\ell_k\in S^1_{\infty}$ are linearly independent and $p\in P_{\infty}$ has a dense $\GL_{\infty}$-orbit. By acting with an invertible element of $E$ as in Example~\ref{ex:copiesS}, we may assume that $\ell_i=x_i$. Take $n_0=k$. Similar to induction step in the proof of the previous theorem, there exists an integer $m_0\geq k$ and a linear map $\psi\colon K^{[m_0]-[k]}\to K^{n_0}$ such that the linear map $\phi_0=\id_k+\,\psi\colon K^k\oplus K^{[m_0]-[k]}\to K^{n_0}$ satisfies $P(\phi_0)p_{m_0}=q_{n_0}=0$. We now proceed as in the proof of the theorem with these $m_0,n_0,\phi_0$ to find the result.
\end{re}

\begin{proof}[Proof of Theorem~\ref{thm:MainII}, existence
of $p$]
The existence of a minimal $p$ among all elements with a dense
$\GL_{\infty}$-orbit follows directly from Propositions~\ref{prop:Dense}
and~\ref{prop:Specialising}.
\end{proof}

\subsection{Maximal tensors} \label{ssec:maximal_class}

Next, we construct maximal elements with respect to $\preceq$ of $P_{\infty}$ for any pure polynomial functor $P$. We start with $n$-way tensors, then do Schur functors and finally general polynomial functors. Let $d\geq 1$ be an integer and let $T^d$ be the polynomial functor sending $V\mapsto V^{\otimes d}$.

\begin{lm}
There exists a tensor $r_d\in T^d_{\infty}$ such that $p\preceq r_d$ for all $p\in T^d_{\infty}$.
\end{lm}
\begin{proof}
For $d=1$, we know that the element $r_1:=x_1\in T^1_{\infty}$ satisfies $p\preceq r_1$ for all $p\in T^1_{\infty}$. Now suppose that $d\geq 2$ and that $r_{d-1}=r_{d-1}(x_1,x_2,\ldots)\in T^{d-1}_{\infty}$ satisfies $p\preceq r_{d-1}$ for all $p\in T^{d-1}_{\infty}$. We define a $r_d\in T^d_{\infty}$ satisfying $p\preceq r_d$ for all $p\in T^d_{\infty}$.

For $j\in\{1,\ldots,d\}$, we define the map $-\otimes_j-\colon T^1_{\infty}\times T^{d-1}_{\infty}\to T^d_{\infty}$ as the inverse limit of the bilinear maps $-\otimes_j-\colon V\times V^{\otimes d-1}\to V^{\otimes d}$ such that $v_j\otimes_j(v_1\otimes \cdots \otimes v_{j-1}\otimes v_{j+1}\otimes\cdots\otimes v_d)=v_1\otimes\cdots\otimes v_d$ for all finite-dimensional vector space $V$ and all vectors $v_1,\ldots,v_d\in V$. Now, we take
\[
r_d:=\sum_{i=1}^{\infty}\sum_{j=1}^dx_{\iota(i,j,1)}\otimes_j r_{d-1}(x_{\iota(i,j,2)},x_{\iota(i,j,3)},\ldots)
\]
where $\iota\colon\NN\times \{1,\ldots,d\}\times\NN\to\NN$ is any injective map. We claim that $p\preceq r_d$ for all $p\in T^d_{\infty}$. Indeed, any such $p$ can we written as
\[
p=\sum_{i=1}^{\infty}\sum_{j=1}^dx_i\otimes_j p_i(x_i,x_{i+1},\ldots)
\]
with $p_1,p_2,\ldots\in T^{d-1}_{\infty}$ and by assumption
we can specialise $r_{d-1}$ to $p_i$ using an element of $E$
for all~$i$. Combined, this yields a specialisation of $r_d$
to $p$. Note here that $x_{\iota(i,j,1)}\mapsto x_i$ and
$x_{\iota(i,j,k)}\mapsto\ell_{i,j,k}$ for $k>1$ in such a
way that $x_{\ell}$ occurs, when ranging over $k$, in only
finitely many $\ell_{i,j,k}$ when $i\leq \ell$ and $x_{\ell}$ does not occur in $\ell_{i,j,k}$ when $i>\ell$. This means that the specialisation of $r_d$ to $p$ indeed goes via an element of $E$. So for all $d\geq 1$, the space $T^d_{\infty}$ has a maximal element with respect to $\preceq$.
\end{proof}

\begin{lm}
Let $P$ be a Schur functor of degree $d\geq 1$. Then there exists a tensor $r\in P_{\infty}$ such that $p\preceq r$ for all $p\in P_{\infty}$.
\end{lm}
\begin{proof}
The space $P_{\infty}$ is a direct summand of $T^d_{\infty}$. Let $r$ be the component in $P_{\infty}$ of $r_d$ from the previous lemma. Then $p\preceq r$ for all $p\in P_{\infty}$.
\end{proof}

\begin{prop}\label{prop:max_el}
Let $P$ be a pure polynomial functor. Then there exists a tensor $r\in P_{\infty}$ such that $p\preceq r$ for all $p\in P_{\infty}$.
\end{prop}
\begin{proof}
Write
\[
P=P^{(1)}\oplus\cdots\oplus P^{(k)}
\]
as a direct sum of Schur functors. For each $i\in\{1,\ldots,k\}$, let $r_i=r_i(x_1,x_2,\ldots)\in P^{(i)}_{\infty}$ be a tensor such that $p_i\preceq r_i$ for all $p_i\in P^{(i)}_{\infty}$ and take $r=(r_1(x_1,x_{k+1},\ldots),\ldots,r_k(x_k,x_{2k},\ldots))\in P_{\infty}$. Then $p\preceq r$ for all $p\in P_{\infty}$.
\end{proof}

\begin{proof}[Proof of Theorem~\ref{thm:MainII}, the existence of $r$]
This follows directly from Proposition~\ref{prop:max_el}.
\end{proof}

\section{Further examples}\label{sec:examples}
In this section we give more examples: we prove that tensors in $P_{\infty}$ with a dense $\GL_{\infty}$-orbit for a single equivalence class when $P$ has degree $\leq 2$, we compare candidates for minimal tensors in a direct sum of $S^d$'s of distinct degrees and we construct maximal elements in $P_{\infty}$ for all $P$ with $P(\{0\})=\{0\}$.

\subsection{Polynomial functors of degree $\leq 2$}

\begin{ex}
Take $P = S^1 \oplus S^1$. Then a pair $(v,w)\in S^1_{\infty}\oplus S^1_{\infty}$ has one of the following forms:
\begin{enumerate}
\item the pair $(v,w)$ with $v,w\in S^1_{\infty}$ linearly independent vectors;
\item the pair $(\lambda u,\mu u)$ with $u\in S^1_{\infty}$ nonzero and $[\lambda:\mu]\in\PP^1$; or
\item the pair $(0,0)$.
\end{enumerate}
In the first case, the pair $(v,w)$ has a dense $\GL_{\infty}$-orbit and is equivalent to $(x_1,x_2)$. When $\mu v-\lambda w=0$ for some $\lambda,\mu\in K$, then this also holds for all specialisations of $(v,w)$. So the poset of equivalence classes is given by:
\begin{center}
\begin{tikzpicture}
	\node (1) at (0,2) {$(x_1,x_2)$};	
	\node (2) at (0,1) {$\PP^1$};
	\node (3) at (0,0) {$(0,0)$};

    \draw [thick] (1) -- (2);
    \draw [thick] (2) -- (3);
\end{tikzpicture}
\end{center}
where a point $[\lambda:\mu]\in \PP^1$ corresponds to the class of $(\lambda u, \mu u)$ with $u \in S^1_{\infty}$ nonzero and all points in $\PP^1$ are incomparable.
\end{ex}

\begin{ex}\label{ex:Quadrics}
Take $P=S^2$. By Proposition~\ref{prop:Specialising} each infinite quadric
\[
p=\sum_{1 \leq i \leq j} a_{ij} x_i x_j
\]
of infinite rank specialises to the quadric $q=x_1x_2 + x_3x_4 + \ldots$
via a suitable linear change of coordinates. Here each variable is only allowed to occur in only finitely many of the linear forms that $x_1,x_2,\ldots$ are substituted by.
Conversely, it is not difficult to see that $q$ specialises
to $p$ as well by applying the following element of $E$:
\[
\begin{pmatrix}
1 & a_{11} & 0 & 0      & 0 & 0 & \cdots\\
0 & a_{12} & 1 & a_{22} & 0 & 0 & \cdots\\
0 & a_{13} & 0 & a_{23} & 1 & a_{33} &\cdots\\
0 & a_{14} & 0 & a_{24} & 0 & a_{34} &\cdots\\
\vdots & \vdots & \vdots & \vdots & \vdots & \vdots &
\end{pmatrix}.
\]
We conclude that the infinite-rank quadrics form a single
equivalence class under $\simeq$
and that the rank function is an isomorphism from the poset of equivalence classes to the well-ordered
set $\{0,1,2,\ldots,\infty\}$.
\end{ex}

\begin{ex}\label{ex:AlternatingTensors}
Take $P=\Wedge^2$. By Proposition~\ref{prop:Specialising} each infinite alternating tensor
\[
p=\sum_{1 \leq i < j} a_{ij} x_i\land x_j
\]
of infinite rank specialises to $q=x_1\land x_2 + x_3\land x_4 + \ldots$. And, $q$ specialises to $p$ as well by applying the following element of $E$:
\[
\begin{pmatrix}
1 & 0 & 0 & 0      & 0 & 0 & \cdots\\
0 & a_{12} & 1 & 0 & 0 & 0 & \cdots\\
0 & a_{13} & 0 & a_{23} & 1 & 0 &\cdots\\
0 & a_{14} & 0 & a_{24} & 0 & a_{34} &\cdots\\
\vdots & \vdots & \vdots & \vdots & \vdots & \vdots &
\end{pmatrix}.
\]
As before, we conclude that the infinite-rank alternating
tensors form a single $\simeq$-equivalence class and that the rank function is an isomorphism from the poset of equivalence classes to the well-ordered set $\{0,1,2,\ldots,\infty\}$.
\end{ex}

\begin{ex}\label{ex:degreeleq2}
Take $P=(S^1)^{\oplus a}\oplus (S^2)^{\oplus b}\oplus (\Wedge^2)^{\oplus c}$ for integers $a,b,c\geq0$. By Remark~\ref{re:addingS1s}, any tuple in $P_{\infty}$ with a dense $\GL_{\infty}$-orbit specialises to the tuple
\begin{align*}
(x_1,\ldots,x_a,y_1y_2+y_{2b+1}y_{2b+2}+\ldots,&\ldots,y_{2b-1}y_{2b}+y_{4b-1}y_{4b}+\ldots,\\
&z_1\land z_2+z_{2c+1}\land z_{2c+2}+\ldots,\ldots,z_{2c-1}\land z_{2c}+z_{4c-1}\land z_{4c}+\ldots)
\end{align*}
where $y_{2ib+j}=x_{a+2ib+2ic+j}$ for $i\geq0$ and $1\leq
j<2b$ and $z_{2ic+j}=x_{a+2(i+1)b+2ic+j}$ for $i\geq0$ and
$1\leq j<2c$. By the previous examples, each of the entries
in this latter tuple independently specialises to any tensor
in the same space. So the entire tuple also specialises to
any other tuple in $P_{\infty}$. So the tuple with a dense
$\GL_{\infty}$-orbit again form a single
$\simeq$-equivalence class.
\end{ex}

\subsection{Non-homogeneous polynomial functors}
The proof of Proposition~\ref{prop:Specialising} relies on the fact that $P$ is homogeneous. Apart from the slight generalisation from Remark~\ref{re:addingS1s}, we don't know if such a result holds in a more general setting.

\begin{que}
Take $P=S^2 \oplus S^3$. Does there exist a tensor $q\in P_{\infty}$ with a dense $\GL_{\infty}$-orbit such that $q\preceq p$ for all $p\in P_{\infty}$ with a dense $\GL_{\infty}$-orbit?
\end{que}

The next example compares different candidates for such a minimal element.

\begin{ex}
Take $P = S^{d_1} \oplus S^{d_2} \oplus \cdots \oplus S^{d_k}$ with $1<d_1<d_2<\cdots<d_k$.
By \cite[Lemma 4.5.3]{B:thesis}, an element $(f_1, \cdots, f_k)\in P_{\infty}$ has dense $\GL_\infty$-orbit if and only if $f_i \in S^{d_i}_{\infty}$ has dense $\GL_\infty$-orbit for all $i = 1, \ldots, k$. In particular, the elements
\[
q =(q^{(1)},\ldots,q^{(k)})= (x_1^{d_1} + x_2^{d_1} + \ldots, \ldots, x_1^{d_k} + x_2^{d_k} + \ldots)
\]
and
\[
p=(p^{(1)},\ldots,p^{(k)})= (x_1^{d_1} + x_{k+1}^{d_1} + \ldots, \cdots, x_k^{d_k} + x_{2k}^{d_k} + \ldots)
\]
have dense $\GL_\infty$-orbits. Clearly $q\preceq p$. By Corollary~\ref{cor:surjective_infty}, there exists an $n\geq 1$ and linear forms $\ell_1,\ldots,\ell_n$ in $x_1,\ldots,x_k$ such that $q_n^{(j)}(\ell_1,\ldots,\ell_n)= x_j^{d_j}$ for $j=1,\ldots,k$. Take
\[
\ell_{hn+i}=\ell_i(x_{hn+1},\ldots,x_{hn+n})
\]
for $h\geq 1$ and $i\in\{1,\ldots,k\}$. Then we see that $q_n^{(j)}(\ell_{hn+1},\ldots,\ell_{hn+n})= x_{hn+j}^{d_j}$ for $j=1,\ldots,k$. So since
\[
q^{(j)}=q_n^{(j)}+q_n^{(j)}(x_{n+1},\ldots,x_{2n})+\ldots
\]
we see that $q^{(j)}(\ell_1,\ell_2,\ldots)=p^{(j)}$. Let $A$ be the $k\times n$ matrix corresponding to $\ell_1,\ldots,\ell_n$ and take
\[
e:=\begin{pmatrix}A\\&A\\&&\ddots\end{pmatrix}\in E
\]
Then $P(e)q^{(j)}=q^{(j)}(\ell_1,\ell_2,\ldots)$. So $p\preceq q$. Hence $p\simeq q$.
\end{ex}

\end{document}